\documentclass[a4paper,11pt,reqno]{amsart}
\usepackage[cyr]{aeguill}
\usepackage[pdfdisplaydoctitle=true,
      colorlinks=true,
      urlcolor=blue,
      citecolor=blue,
      linkcolor=blue,
      pdfstartview=FitH,
      pdfpagemode=None,
      bookmarksnumbered=true]{hyperref}

\newtheorem{thm}{Theorem}
\newtheorem{cor}[thm]{Corollary}  
\newtheorem{lem}[thm]{Lemma}    
\newtheorem{prop}[thm]{Proposition} 
\newtheorem{defi}[thm]{Definition} 

\newtheorem*{pb}{Problem}

\theoremstyle{definition}

\theoremstyle{remark}
\newtheorem{rem}{Remark}

\numberwithin{equation}{section}  
\numberwithin{figure}{section}   

\newcommand{\CC}{{\mathbb{C}}}
\newcommand{\NN}{{\mathbb{N}}}
\newcommand{\PP}{{\mathbb{P}}}
\newcommand{\RR}{{\mathbb{R}}}

\newcommand{\eps}{\varepsilon}
\newcommand{\f}{\varphi}
\newcommand{\p}{\psi}

\newcommand{\dbar}{\bar{\partial}}
\newcommand{\zbar}{\bar{z}}

\newcommand{\norm}[1]{\left\Vert#1\right\Vert}
\newcommand{\abs}[1]{\left\vert#1\right\vert}
\newcommand{\set}[1]{\left\{#1\right\}}

\DeclareMathOperator{\Hess}{Hess}

\DeclareMathOperator{\Ric}{Ric}
\DeclareMathOperator{\tr}{tr}

\hypersetup{
  pdftitle={K\"{a}hler-Einstein fillings}, 
  pdfauthor={Vincent Guedj, Boris Kolev, Nader Yeganefar}, 
  pdfsubject={Complex Monge-Amp\`{e}re equation, K\"{a}hler-Einstein metrics, Local Moser-Trudinger inequality, pseudoconvex domains}, 
  pdfkeywords={MSC 2010: 32W20, 32Q20}, 
  pdflang=EN 
  }

\begin{document}

\title{K\"{a}hler-Einstein fillings}

\author{Vincent Guedj}
\address{Institut Universitaire de France et Institut de Math\'{e}matiques de Toulouse, Universit\'{e} Paul Sabatier,
31062 Toulouse cedex 09, France}
\email{vincent.guedj@math.univ-toulouse.fr}

\author{Boris Kolev}
\address{LATP, CNRS \& Universit\'{e} de Provence, 39 Rue F. Joliot-Curie, 13453 Marseille Cedex 13, France}
\email{kolev@cmi.univ-mrs.fr}

\author{Nader Yeganefar}
\address{LATP, Universit\'{e} de Provence, 39 Rue F. Joliot-Curie, 13453 Marseille Cedex 13, France}
\email{Nader.Yeganefar@cmi.univ-mrs.fr}


\subjclass[2010]{32W20, 32Q20}%
\keywords{Complex Monge-Amp\`{e}re equation, K\"{a}hler-Einstein metrics, Local Moser-Trudinger inequality, pseudoconvex domains}%

\date{2011-11-22}%


\begin{abstract}
We show that on an open bounded smooth strongly pseudoconvex subset of $\CC^{n}$, there exists a K\"{a}hler-Einstein metric with positive Einstein constant, such that the metric restricted to the Levi distribution of the boundary is conformal to the Levi form. To achieve this, we solve an associated complex Monge-Amp\`{e}re equation with Dirichlet boundary condition. We also prove uniqueness under some more assumptions on the open set.
\end{abstract}

\maketitle


\tableofcontents
\clearpage


\section{Introduction}
\label{sec:intro}

The study of Einstein metrics is an important and classical subject in Riemannian geometry, see \cite{Bes08}. The most popular framework is that of complete manifolds, either compact (without boundary) or noncompact. However, Einstein metrics on compact manifolds with boundary have also been investigated more recently, mainly in two directions which we now describe.

The first direction is that of conformally compact manifolds. Here, one starts with a compact manifold $M$ with boundary $\partial M$. A complete Einstein metric on (the interior of) $M$ is called conformally compact if after a suitable conformal transformation, it can be extended smoothly up to the boundary (think of the ball model of real hyperbolic space, or look at \cite{Biq00} for the precise definition). This extension is not unique, but different extensions are easily seen to induce Riemannian metrics on the boundary which are in the same conformal class, called the conformal infinity of the conformally compact metric. One of the basic questions is then: a conformal class being fixed on the boundary, is it possible to find a conformally compact Einstein metric on $M$ whose conformal infinity is the given conformal class? One then hopes to get links between the geometric properties of the inner metric and the conformal properties of the boundary; for more on this very a
 ctive research area, the reader may consult e.g. \cite{Biq00,AH08}.

We now come to the second direction, which has been explored far less than the first one and is more closely related to our present work. One starts again with a compact manifold $M$ with boundary, and fixes some geometric structure on the boundary (for example a metric). The problem is then to find an Einstein metric on $M$ which is smooth up to the boundary, and which induces the given geometric structure on $\partial M$.
Assume for example that there is an Einstein metric on $M$ with pinched negative curvature such that the boundary is convex and umbilical and let $h_{0}$ be the induced metric on $\partial M$. If $h$ is a metric on $\partial M$ which is sufficiently close to $h_{0}$, it has been shown in \cite{Sch01} that there is an Einstein metric on $M$ with negative Einstein constant such that the induced metric on $\partial M$ is $h$. One of the interesting questions, which has not been fully clarified yet, is to know what ``right'' geometric structure has to be fixed on the boundary. \cite{And08} considers the Dirichlet problem as in \cite{Sch01} (given a metric $h$ on $\partial M$, can one find an Einstein metric on $M$ inducing $h$ on $\partial M$?), studies the structure of the space of solutions and observes that this Dirichlet problem is not a well-posed elliptic boundary value problem. On the other hand, if one prescribes the metric and the second fundamental form of $\partial M$,
  then any Einstein metric on $M$ is essentially unique by \cite{AH08}.

The main purpose of this article is to investigate similar questions in the context of compact K\"{a}hler manifolds with boundary. Let $M$ be a compact K\"{a}hler manifold with \emph{strongly pseudoconvex} boundary $\partial M$. The latter is a CR manifold whose geometric properties are encoded by the (conformal class of its) \emph{Levi form}, a positive definite Hermitian form defined on the \emph{Levi distribution} $T_{\CC}(\partial M)$ (the family of maximal complex subspaces within the real tangent bundle).
The question we address is the following:

\begin{pb}
Can one find a K\"{a}hler-Einstein metric $\omega$ on $M$ such that its restriction to the Levi distribution is conformal to the Levi form on $T_{\CC}(\partial M)$?
\end{pb}

To simplify we restrict ourselves in the sequel to studying the case of a strongly pseudoconvex bounded open subset $\Omega$ of $\CC^{n}$. One can then always make a conformal change of the Levi form so that the pseudo-Hermitian Ricci tensor (introduced by Webster) is a scalar multiple of the Levi form, i.e. $\partial \Omega$ is pseudo-Einstein (see \cite{Lee88}). Our problem is thus intimately related to the Riemannian questions recalled above.

It is well known that finding a K\"{a}hler-Einstein metric is equivalent to solving a complex Monge-Amp\`{e}re equation. More specifically,
letting $\mu$ denote the Lebesgue measure in $\CC^{n}$ normalized such that $\mu (\Omega)=1$, we will be interested in the following Dirichlet problem : find a smooth strictly plurisubharmonic function $\varphi$ on $\Omega$ which vanishes on the boundary $\partial \Omega$ and satisfies
\begin{equation*}
  (dd^{c}\varphi )^{n} = \frac{e^{-\eps\varphi}\mu}{\int_\Omega e^{-\eps\varphi}\, d\mu} \quad \text{in} \quad \Omega,
\end{equation*}
where $\eps\in\{0,\pm 1\}$ is a fixed constant. If $\varphi$ is a solution of this problem, then it is easy to see that $dd^{c}\varphi$ is a K\"{a}hler-Einstein metric with the sign of the Einstein constant given by $\eps$, and moreover its restriction to the Levi distribution is conformal to the Levi form on $T_{\CC}(\partial \Omega)$ (see section~\ref{sec:geometric_context} for more details on this). Actually, if $\eps =0, -1$, then the Monge-Amp\`{e}re equation above has always a solution by Theorem 1.1 in \cite{CKNS85}, so that we will only consider the positive curvature case corresponding to $\eps = 1$. Our main result is

\begin{thm}\label{thm:thmA}
Let $\Omega \subset \CC^{n}$ be a bounded smooth strongly pseudoconvex domain. Then the complex Monge-Amp\`{e}re problem
\begin{equation*}
   (MA)\quad (dd^{c} \varphi)^{n}=\frac{e^{-\varphi}\mu}{\int_\Omega e^{-\varphi}\, d\mu}\, \textrm{in $\Omega$, and $\varphi_{\vert \partial \Omega}=0$}
\end{equation*}
has a strictly plurisubharmonic solution which is smooth up to the boundary.
\end{thm}

By the considerations of section~\ref{sec:geometric_context}, a consequence of this theorem is that our geometrical problem has a solution:

\begin{cor}\label{cor:corB}
Let $\Omega \subset \CC^{n}$ be a bounded smooth strongly pseudoconvex domain. Then there is a smooth (up to the boundary) K\"{a}hler-Einstein metric on $\Omega$ with positive Einstein constant such that the restriction of the metric to the Levi distribution of $\partial \Omega$ is conformal to the Levi form.
\end{cor}

Let us now say a few words about the proof of our main theorem. We will use a Ricci inverse iteration procedure, as described first in the compact K\"{a}hler setting by \cite{Kel09} and \cite{Rub08}, whereas related results have recently been obtained in \cite{BB11,Ceg11} by other interesting approaches. More precisely, fix any smooth strictly plurisubharmonic function $\f_{0}$ on $\Omega$ which vanishes on the boundary, and for $j\in \NN$, let $\f_{j}$ be the unique strictly plurisubharmonic solution of the Dirichlet problem
\begin{equation*}
   (dd^{c} \f_{j+1})^{n}=\frac{e^{-\varphi_{j}}\mu}{\int_\Omega e^{-\varphi_{j}}\, d\mu}\, \textrm{in $\Omega$, and ${\varphi_{j+1}}_{\vert \partial \Omega}=0$},
\end{equation*}
whose existence is guaranteed by \cite{CKNS85}. We will then show that $(\f_{j})$ is bounded in $C^\infty (\bar{\Omega})$, so that a subsequence converges in $C^\infty (\bar{\Omega})$ to a smooth function which is seen to be a solution of $(MA)$. To prove this boundedness in $C^\infty$, we proceed in several steps. First, there is a well-known functional $\mathcal{F}$, defined on the space of plurisubharmonic functions, such that a function $\f$ solves $(MA)$ if and only if $\f$ is a critical point of $\mathcal{F}$ (see subsection~\ref{subsec:properness}). A key result is that this functional is proper in the sense of Proposition~\ref{prop:proper}. This properness result is in turn a consequence of a local Moser-Trudinger inequality (see Theorem~\ref{thm:MT}, and also the recent independent results of \cite{BB11,Ceg11}). Next, we show that the sequence $(\mathcal{F}(\f_{j}))$ is bounded, so that by properness, the sequence $(\f_{j})$ has to live in some compact set. Here, com
 pactness is for the $L^1$-topology in the class of plurisubharmonic functions with finite energy introduced in \cite{BGZ09}. Standard results from pluripotential theory then show that $(\f_{j})$ is uniformly bounded. To get boundedness in $C^\infty$, we will finally prove higher order a priori estimates, along the lines of \cite{CKNS85}.

\smallskip

Now, let us deal with the uniqueness problem. For this, we impose some restrictions on $\Omega$. First, we assume that $\Omega$ contains the origin and is circled; this means that $\Omega$ is invariant by the natural (diagonal) $S^1$-action on $\CC^{n}$. Next, if $\varphi$ is a $S^1$-invariant solution of the Monge-Amp\`{e}re equation with Dirichlet boundary condition, we will say that $\Omega$ is (strictly) $\varphi$-convex if $\Omega$ is (strictly) convex in the Riemannian sense for the metric $dd^{c}\varphi$. Note that being $\varphi$-convex has \textit{a priori} nothing to do with being convex in the usual Euclidean sense in $\CC^{n}$. We will prove

\begin{thm}\label{thm:thmC}
Let $\Omega \subset \CC^{n}$ be a bounded smooth strongly pseudoconvex domain which is circled. Let $\varphi$ be a smooth $S^1$-invariant strictly plurisubharmonic solution of the complex Monge-Amp\`{e}re problem $(MA)$. If $\Omega$ is strictly $\varphi$-convex, then $\varphi$ is the unique $S^1$-invariant solution of $(MA)$.
\end{thm}

Observe that a $S^1$-invariant solution always exists, as follows from the proof of
Theorem \ref{thm:thmA}: it suffices to start with an initial datum $\f_0$ which is $S^1$-invariant, the
approximants $\f_j$ will also be $S^1$-invariant (by the uniqueness part of \cite{CKNS85}), hence so
is any cluster value.

\begin{rem}
In the proof of this theorem, we will see that we can replace the $\varphi$-convexity hypothesis by a spectral assumption. Namely, if the first eigenvalue of the Laplace operator (of the metric $\omega^\f=dd^{c}\varphi$) with Dirichlet boundary condition is strictly bigger than $1$, then $(MA)$ has a unique solution. By Corollary 1.2 in \cite{GKY11}, the condition on the Ricci curvature of $\omega^\f$ and the strict $\varphi$-convexity imply this desired spectral estimate. However, \cite{GKY11}[Proposition 4.1] shows that this estimate may fail if $\Omega$ is merely strongly pseudoconvex.
\end{rem}

To prove Theorem~\ref{thm:thmC}, we follow the approach proposed by Donaldson in the compact (without boundary) setting (see \cite{Don99,BBGZ09}). The heuristic point of view is the following. The space of all plurisubharmonic functions on $\Omega$ which vanish on the boundary may be seen as an infinite dimensional manifold with a natural Riemannian structure. In the $S^1$-invariant case, we may use a convexity result of Berndtsson \cite{Ber06} to show that the functional $\mathcal F$ is concave along geodesics of this space. As a consequence, we show that $S^1$-invariant solutions of $(MA)$ coincide with $S^1$-invariant maximizers of the functional $\mathcal{F}$. Now, if $\f$ and $\p$ are two $S^1$-invariant solutions of $(MA)$, then there exists a geodesic $(\Phi_{t})_{0\leq t\leq 1}$ in the space of K\"{a}hler potentials on $\Omega$ vanishing on the boundary which joins $\f$ to $\p$. Therefore, the function $t\mapsto \mathcal{F}(\Phi_{t})$, being concave and attaining its
 maximum at $t=0$ and $t=1$, must be constant. In particular, its derivative vanishes, which implies that $\dot{\Phi}_{0}$ has to satisfy a PDE involving the Laplacian of the metric $dd^{c}\f$ (see equation \eqref{equ:elliptic_equation} below). If $\Omega$ is $\varphi$-convex, or more generally if the spectral hypothesis alluded to above is satisfied, then the only solution of this PDE is zero, so that $\dot{\Phi}_{0}$ vanishes identically. From this, we may deduce that $(\Phi_{t})$ is a constant geodesic, hence $\f=\p$. Note that in the above argument, we have implicitly assumed that $(\Phi_{t})$ is smooth, which may not be the case. For general continuous geodesics, the proof needs some modifications which will be given in section~\ref{sec:conclusion}.

\smallskip

This uniqueness result has the following application. In \cite[Conjecture 7.5]{BB11}, it is conjectured that if $B$ is a ball in $\CC^{n}$, then any solution of $(MA)$ has to be radial. Theorem~\ref{thm:thmC} shows that this is the case among $S^1$-invariant solutions if the radius of the ball is not too large. Indeed, let $B\subset \CC^{n}$ be the ball of radius $R>0$ centered at $0$. Consider the radial function
\begin{equation*}
  \f = \frac{n+1}{\pi}\left[\log{\sqrt{1+\norm{z}^{2}}}-\log{\sqrt{1+R^{2}}}\right].
\end{equation*}
In an affine chart, $\f$ is the potential of the Fubini-Study metric on complex projective space $\PP^{n}(\CC)$, normalized to satisfy $(MA)$ on $B$. Note that $B$ may also be considered as a ball in $\PP^{n}(\CC )$, whose radius $R_{FS}$ with respect to the Fubini Study metric is
\begin{equation*}
  R_{FS}=\sqrt{\frac{n+1}{\pi}}\arctan{R}.
\end{equation*}
The diameter of $\PP^{n}(\CC)$ is then
\begin{equation*}
  D_{FS}=\sqrt{\pi (n+1)}/2.
\end{equation*}
If $R_{FS}< D_{FS}/2$, then $B$ is strictly convex in $\PP^{n}(\CC )$, that is $B$ is strictly $\f$-convex (this is a well-known result, see for example the proof of \cite[Proposition 4.1]{GKY11}). By Theorem~\ref{thm:thmC}, $\f$ is the unique $S^1$-invariant solution of $(MA)$, so that all such solutions are radial. We have thus proved

\begin{cor}\label{cor:DDL}
Let $B$ be a ball in $\CC^{n}$ of radius $0<R<1$.
Then there is a unique $S^1$-invariant solution to $(MA)$ on $B$, and this solution is radial.
\end{cor}

The plan of the paper is as follows. In section~\ref{sec:geometric_context}, we gather some well-known facts on the geometry of pseudoconvex domains and show how our geometrical problem is related to the analytical problem of solving a complex Monge-Amp\`{e}re equation with Dirichlet boundary condition. In section~\ref{sec:energy_estimates}, we prove a local Moser-Trudinger inequality and use it to prove a properness result for the functional $\mathcal{F}$. In section~\ref{sec:higher_order_estimates}, we deal with the regularity problem of solutions of $(MA)$, by getting higher order a priori estimates. This will allow us to prove Theorem~\ref{thm:thmA} in subsection~\ref{subsec:conclusion}. In section~\ref{sec:variational_characterization}, we obtain a variational characterization of solutions of $(MA)$ in the $S^1$-invariant case. Indeed, we show that $S^1$-invariant solutions of $(MA)$ are not only critical points of the functional $\mathcal{F}$, but are exactly maximizers of $\mathcal{F}$. Then we proceed to prove Theorem~\ref{thm:thmC}. In section~\ref{sec:conclusion}, we comment on the difficulty of solving $(MA)$ by the usual continuity method, and finally discuss the optimality of constants in the Moser-Trudinger inequality.

\medskip

\noindent {\bf Acknowledgements.} It is a pleasure to thank Robert Berman and Bo Berndtsson for stimulating discussions related to their joint work \cite{BB11}. We would like also to thank Thierry Gallou\"{e}t and Marc Herzlich for helpful discussions during the preparation of this paper.


\section{Geometric context}
\label{sec:geometric_context}


\subsection{The conformal class of the Levi form}
\label{subsec:clevi}

Let $\Omega \subset \CC^{n}$ be a bounded domain with smooth boundary. Fix a defining function $\rho : \CC^{n} \to \RR$ for the boundary $\partial \Omega$, i.e. $\rho$ is a smooth function satisfying
\begin{equation*}
  \Omega = \set{\rho < 0}, \qquad \partial \Omega = \set{\rho = 0} ,
\end{equation*}
and $d\rho$ does not vanish on $\partial \Omega$. Such a function $\rho$ is not unique, but if $\tilde{\rho}$ is another defining function for the boundary, then there is a smooth positive function $u$ such that $\tilde{\rho} =u\rho$.

Let now $x\in \partial \Omega$ be a fixed point, and denote by $H_x$ the maximal complex subspace of the tangent space $T_x\partial \Omega$. If $J$ denotes the complex structure on $\CC^{n}$ (which is just multiplication by $\sqrt{-1}$), then we have
\begin{equation*}
  H_x = \set{v\in T_x\partial \Omega ; \; Jv\in T_x\partial \Omega}.
\end{equation*}
The subspace $H_x$ has real dimension $2n-2$, and as $x$ varies, we get a distribution $H\subset T\partial \Omega$, called the \emph{Levi distribution}. If $(z_{1}, \ldots , z_{n})$ are the coordinates on $\CC^{n}$, then it is easy to see that
\begin{equation}\label{eq:contact}
  H_x = \set{v=(v_{1},\ldots ,v_{n})\in \CC^{n} ; \; \sum_{i=1}^{n}\frac{\partial \rho}{\partial z_{i}}(x)v_{i}=0}.
\end{equation}
The \emph{Levi form} is the Hermitian form defined for $v,w\in H_x$ by
\begin{equation*}
  L_x(v,w) = \sum_{i,j}^{n}\frac{\partial^{2}\rho}{\partial z_{i}\partial \zbar_{j}}(x)v_{i}\bar{w}_{j}.
\end{equation*}
It is clear from this expression that the Levi form actually depends on $\rho$, so talking about \emph{the} Levi form is a slight abuse. However, if $\tilde{\rho} = u \rho$ is another defining function for the boundary (with $u$ a smooth positive function), then we have
\begin{equation*}
   \frac{\partial^{2}\tilde{\rho}}{\partial z_{i}\partial \zbar_{j}} = u\frac{\partial^{2}\rho}{\partial z_{i}\partial \zbar_{j}} + \frac{\partial u}{\partial z_{i}}\frac{\partial \rho}{\partial \zbar_{j}} + \frac{\partial u}{\partial \zbar_{j}}\frac{\partial \rho}{\partial z_{i}}+ \rho \frac{\partial^{2}u}{\partial z_{i}\partial \zbar_{j}}.
\end{equation*}
Moreover, by using the characterization \eqref{eq:contact} of $H$ and the fact that $\rho =0$ on $\partial \Omega$, we infer, denoting by $\tilde{L}$ the Levi form corresponding to $\tilde{\rho}$, that
\begin{equation*}
  \tilde{L}=uL.
\end{equation*}
In other words, the Levi forms corresponding to different defining functions for the boundary differ only by a conformal factor. Thus, the geometrically interesting object on the boundary is the conformal class of the Levi form.

We say that $\Omega$ is \emph{strongly pseudoconvex} if the Levi form is a positive definite Hermitian form at each point of $\partial \Omega$. Our previous discussion shows that this notion does not depend on the choice of a defining function for the boundary. Note also that by changing $\rho$ to $e^{c\rho}-1$, where $c>0$ is a large enough positive constant, we may assume that the Levi form is positive definite
in a neighborhood of  $\overline{\Omega}$, and not only on the Levi distribution.


\subsection{K\"{a}hler metrics}
\label{subsec:kahler_metric}

We give here a brief review of K\"{a}hler metrics, mainly to set up some notations and conventions. For more details and proofs, the reader may consult e.g. \cite{Mor07}. Although we will be dealing with domains in $\CC^{n}$ in the sequel, we consider a general complex manifold $X$ of complex dimension $n$, and denote by $J$ its complex structure.

\subsubsection{K\"{a}hler form}

A Riemannian metric $g$ on $X$ is called \emph{Hermitian} if it is $J$-invariant, i.e. $g(J\cdot, J\cdot)=g(\cdot,\cdot)$. The $\CC$-bilinear extension of $g$ to the complexified tangent bundle $TX\otimes \CC$ will also be denoted by the same symbol $g$. The \emph{fundamental form} associated to $g$ is the real $(1,1)$-form $\omega$ defined by
\begin{equation*}
  \omega (\cdot , \cdot) =g(J\cdot , \cdot).
\end{equation*}
The metric $g$ is called a \emph{K\"{a}hler metric} if $\omega$ is a closed differential form; $\omega$ is then referred to as the \emph{K\"{a}hler form} of $g$. It can be shown that $g$ being a K\"{a}hler metric is equivalent to the complex structure $J$ being parallel with respect to the Levi-Civita connection of $g$.

Let $(z_{1}, \dotsc , z_{n})$ be local complex coordinates, and let
\begin{equation*}
  z_{1} = x_{1} + \sqrt{-1}y_{1}, \quad \dotsc \quad , \quad z_{n} = x_{n} + \sqrt{-1}y_{n}
\end{equation*}
be the decomposition giving the corresponding real coordinates. As usual, for $i=1, \dotsc , n$, we set
\begin{equation*}
   \frac{\partial}{\partial z_{i}}=\frac{1}{2}(\frac{\partial}{\partial x_{i}}-\sqrt{-1} \frac{\partial}{\partial y_{i}} ) ,\quad \frac{\partial}{\partial \zbar_{i}}=\frac{1}{2}(\frac{\partial}{\partial x_{i}}+\sqrt{-1} \frac{\partial}{\partial y_{i}} ),
\end{equation*}
\begin{equation*}
  dz_{i}=dx_{i}+\sqrt{-1}dy_{i} ,\quad d\zbar_{i}=dx_{i}-\sqrt{-1}dy_{i},
\end{equation*}
and for $i,j=1,\dotsc , n$,
\begin{equation*}
  g_{i\bar{j}} = g(\frac{\partial}{\partial z_{i}} , \frac{\partial}{\partial \zbar_{j}}).
\end{equation*}
Then the K\"{a}hler form is given locally by
\begin{equation*}
  \omega =\sqrt{-1} \sum_{i,j=1}^{n} g_{i\bar{j}} dz_{i}\wedge d\zbar_{j}.
\end{equation*}
Note that on $\CC^{n}$, we have $g_{i\bar{j}}=\delta_{ij}/2$ for the canonical Euclidean metric.

\subsubsection{Ricci curvature form}

We denote by $r$ the Ricci tensor of $X$ as a Riemannian manifold. The \emph{Ricci form} of $X$, to be denoted by $\Ric{(\omega)}$ or simply $\Ric$, is the $(1,1)$-form associated to $r$, i.e.
\begin{equation*}
  \Ric{(\omega )}(\cdot,\cdot) =r(J\cdot,\cdot).
\end{equation*}
In local holomorphic coordinates, it can be shown that
\begin{equation*}
  \Ric{(\omega )} = -\sqrt{-1}\partial \dbar \log{\det{g_{i\bar{j}}}}.
\end{equation*}
There follows that the Ricci form is a closed form. Moreover, its cohomology class is equal to $2\pi c_{1}(X)$ , where $c_{1}(X)$ is the first Chern class of $X$.
A K\"{a}hler metric $\omega$ on $X$ is called \emph{K\"{a}hler-Einstein} if for some constant $\lambda \in \RR$, we have
\begin{equation*}
  \Ric{(\omega )} = \lambda \omega .
\end{equation*}

\subsubsection{Normalization of $d^{c}$}

We set
\begin{equation*}
  d^{c}=\frac{1}{2\pi \sqrt{-1}}(\partial -\dbar),
\end{equation*}
so that
\begin{equation*}
  \sqrt{-1}\partial \dbar = \pi dd^{c}.
\end{equation*}

This normalization is of common use in complex analytic geometry, having the following advantages: the positive current
$T=dd^{c} \log \norm{z}$ has then Lelong number $1$ at the origin in $\CC^{n}$; moreover the Fubini-Study form
$\omega_{FS}$ writes, in some affine chart $\CC^{n}$,
\begin{equation*}
  \omega_{FS}=dd^{c} \log \sqrt{1 + \norm{z}^{2}}.
\end{equation*}
Its cohomology class thus coincides with that of a hyperplane (as it should), having total volume
\begin{equation*}
  \int_{\PP^{n}} \omega_{FS}^{n}=\int_{\CC^{n}} \left(dd^{c} \log \sqrt{1+\norm{z}^{2}} \right)^{n}=1.
\end{equation*}
Note finally that $\Ric{(\omega_{FS})} = (n+1)\pi\omega_{FS}$.

Likewise, the Laplacian $\Delta$ associated to a K\"{a}hler metric $\omega$ is defined as
\begin{equation*}
  \Delta = \tr{(dd^{c})},
\end{equation*}
where $\tr$ denotes the trace with respect to $\omega$. Hence, we have
\begin{equation*}
  \Delta = -\frac{1}{\pi}\dbar^*\dbar .
\end{equation*}


\subsection{K\"{a}hler-Einstein metrics on strongly pseudoconvex domains}
\label{subsec:pseudoconvex_domain}

Fix $\Omega \subset \CC^{n}$ a bounded strongly pseudoconvex domain.

\subsubsection{Associated complex Monge-Amp\`{e}re equations}

In this section, we show that finding K\"{a}hler-Einstein metrics is equivalent to solving a complex Monge-Amp\`{e}re equation.

We assume first that $\Omega$ is endowed with a K\"{a}hler metric $\omega$ which is smooth up to the boundary, and which satisfies the following normalized Einstein condition:
\begin{equation*}
  \Ric{(\omega )}=\eps\pi \omega,
\end{equation*}
where $\eps \in \{ 0, \pm 1\}$ (the somewhat unusual $\pi$ factor is due to our normalization convention for the $d^{c}$ operator). We choose a smooth potential $\f$ for $\omega$, so that
\begin{equation*}
  \omega =dd^{c}\f.
\end{equation*}
Such a potential is unique up to the addition of a pluriharmonic function on $\Omega$. We are going to see that $\f$ satisfies a complex Monge-Amp\`{e}re equation. As recalled in the previous section, if we denote by $(g_{i\bar{j}})$ the components of the metric in coordinates, then the Ricci form is given by
\begin{equation*}
  \Ric{(\omega)}=-\pi dd^{c} \log{(\det{g_{i\bar{j}})}}.
\end{equation*}
Letting $V_{0}$ be the canonical volume form on $\CC^{n}$, it is easily checked that $\omega^{n}$ is equal to $\det{(g_{i\bar{j}})}V_{0}$, up to a multiplicative constant. Therefore, we have the following intrinsic formula for the Ricci form:
\begin{equation*}
  \Ric{(\omega)}=-\pi dd^{c} \log{\frac{\omega^{n}}{V_{0}}}.
\end{equation*}
The Einstein condition on $\omega$ can then be written
\begin{equation*}
  dd^{c} \left[\log{\frac{(dd^{c} \f)^{n}}{V_{0}}}+\eps\f \right]=0.
\end{equation*}
Thus, there is a pluriharmonic function $h$ such that
\begin{equation*}
  \log{\frac{(dd^{c} \f)^{n}}{V_{0}}}+\eps\f =h,
\end{equation*}
which we may write as a complex Monge-Amp\`{e}re equation
\begin{equation}\label{eq:MA}
  (dd^{c} \f)^{n}=e^{-\eps\f}e^hV_{0}.
\end{equation}

Conversely, if $\f$ is a smooth function satisfying the previous equation for some given pluriharmonic function $h$, and if $\omega =dd^{c}\f$ is positive definite, we let the reader verify that $\omega$ is a K\"{a}hler-Einstein metric with Einstein constant $\eps\pi$.

\subsubsection{Boundary conditions}
\label{subsec:nonpositive}

Let $\rho$ be a boundary defining function for $\Omega$, as described in section~\ref{subsec:clevi}. Recall that $L$ is the Levi form associated to $\rho$. The $(1,1)$-form associated to $L$, that is $L(J\cdot,\cdot)$, is equal to $\pi dd^{c}\rho$ with our normalization conventions. Let now $\f$ be a smooth real valued function defined on $\bar{\Omega}$. On a collar neighborhood $[-\delta , 0]\times \partial \Omega$ of $\partial \Omega$ (where $\delta >0$ is fixed), we can write the expansion of $\f$ in powers of $\rho$ as follows: for all $N\in \NN$,
\begin{equation}\label{eq:DL}
  \f = \f_{0} + \rho \f_{1} + \rho^{2} \f_2 + \dotsb +\rho^{N} \f_{n} + o(\rho^{N}).
\end{equation}
Here, the functions $\f_{i}$ are initially defined on $\{ 0\}\times \partial \Omega\simeq \partial \Omega$, but we can view them as functions defined on the collar neighborhood $[-\delta ,0]\times \partial \Omega$ by setting, with obvious notations, $\f_{i}(\rho ,x)= \f_{i}(0,x)$. Thus, we have for example $\f_{0} =0$ if $\f \vert_{\partial \Omega} =0$. From the expansion \eqref{eq:DL}, we get
\begin{equation*}
  dd^{c} \f = dd^{c} \f_{0} +\f_{1} dd^{c} \rho +d \rho \wedge d^{c} \f_{1}+ (d\f_{1}+2\f_2 d\rho )\wedge d^{c}\rho+O(\rho ).
\end{equation*}
Using the fact that $d\rho =d^{c}\rho = 0$ on the Levi distribution $H$ (see the characterization \eqref{eq:contact} of $H$), the previous expansion implies
\begin{equation*}
  dd^{c}\f \vert_H= dd^{c}\f_{0}\vert_H+\f_{1} dd^{c}\rho.
\end{equation*}
In particular, if $\f_{0}=0$, or more generally if $dd^{c}\f_{0}=0$, then $dd^{c}\f \vert_H$ is conformal to the Levi form.\\

Consider now the following geometrical problem: find a K\"{a}hler-Einstein metric $\omega$ on $\Omega$ such that its restriction to the Levi distribution is conformal to the Levi form. Our previous discussion shows that in order to solve this problem, it is enough to solve the following analytical problem: find a function $\f$ such that
\begin{enumerate}
 \item $dd^{c}\f$ is positive definite,
 \item $\f$ satisfies the Monge-Amp\`{e}re equation \eqref{eq:MA},
 \item $\f$ satisfies the Dirichlet boundary condition on $\partial \Omega$, i.e. $\f \vert_{\partial \Omega}=0$.
\end{enumerate}
Indeed, the form $\omega =dd^{c}\f$ is then a solution to the geometrical problem. Note that in the case of nonpositive Ricci curvature, which corresponds to $\eps =0$ or $-1$ in equation \eqref{eq:MA}, the geometrical problem  always has a solution by \cite[Theorem 1.1]{CKNS85}. We will therefore consider only the positive curvature case ($\eps =1$).


\subsection{The strategy}
\label{subsec:strategy}

In the sequel we let $\Omega=\{ \rho < 0\} \subset \CC^{n}$ be a bounded strongly pseudoconvex domain
and
$\mu$ denote the euclidean Lebesgue volume form in $\CC^{n}$, normalized so that
\begin{equation*}
  \mu(\Omega)=1.
\end{equation*}
We consider the
following Dirichlet problem
\begin{equation*}
(MA) \hskip2cm
(dd^{c} \f)^{n}=\frac{e^{-\f} \mu}{\int_{\Omega} e^{-\f} d \mu}
\; \text{ in } \Omega
\hskip.4cm
\text{ with }
\hskip.4cm
\f_{|\partial \Omega}=0,
\end{equation*}
where $\f$ is strictly plurisubharmonic and $\mathcal{C}^{\infty}$-smooth up to the boundary of $\Omega$.

We are going to solve $(MA)$ by the \emph{an iterative process}, solving for each $j \in \NN$
the Dirichlet problem
\begin{equation*}
(MA)_{j} \hskip1cm
(dd^{c} \f_{j+1})^{n}=\frac{e^{-\f_{j}} \mu}{\int_{\Omega} e^{-\f_{j}} d \mu}
\; \text{ in } \Omega
\hskip.4cm
\text{ with }
\hskip.4cm
{\f_{j+1}}_{|\partial \Omega}=0,
\end{equation*}
where $\f_{0}=\rho$ (we could actually start from any smooth plurisubharmonic initial data
$\f_{0}$ with zero boundary values).

It follows from the work of Cafarelli-Kohn-Nirenberg-Spruck \cite{CKNS85} that the Dirichlet
problem $(MA)_{j}$ admits a unique plurisubharmonic solution $\f_{j}$ which is smooth up to the boundary.
 We are going to show that a subsequence of the sequence $(\f_{j})$ converges in $\mathcal{C}^{\infty}(\bar{\Omega})$
towards a solution $\f$ of $(MA)$.

In a compact setting this approach coincides with the time-one discretization of the K\"{a}hler-Ricci
flow and was first considered by Keller \cite{Kel09} and Rubinstein \cite{Rub08}
(see also \cite{BBEGZ11}).

\begin{rem}
As the proof will show, our result actually holds for any (normalized) volume form $\mu$
and with more general boundary values.
\end{rem}


\section{Energy estimates}
\label{sec:energy_estimates}

We now move on to showing that the sequence $(\f_{j})$ is relatively compact
in $\mathcal{C}^{\infty} (\bar{\Omega})$.
The proof reduces to establishing
\emph{a priori estimates}. We first show that one has a uniform a priori control on the
\emph{energy} of the solutions.


\subsection{Local Moser-Trudinger inequality}

The following local Moser-Trudinger type inequality is of independent interest.\footnote{While
we were finishing the writing of this paper, two preprints appeared \cite{BB11,Ceg11} which propose similar inequalities with different and interesting proofs.}

\begin{thm} \label{thm:MT}
There exists $0<\beta_{n}<1$ and $C>0$ such that for all smooth plurisubharmonic functions $\f$ in $\Omega$ with
$\f_{|\partial \Omega} =0$,
\begin{equation*}
  \int_{\Omega} e^{-\f} d \mu \leq C \exp \left( \beta_{n} \abs{\mathcal{E}(\f)} \right),
\end{equation*}
where $\mathcal{E}(\f)=\frac{1}{n+1} \int_{\Omega} \f \, (dd^{c} \f)^{n}$.
\end{thm}

We refer the reader to \cite{Mos71,Ono82,Tia97,Tia00,CL04,PSSW08} for related results both in a local and global context.
The proof we propose is new and relies on pluripotential techniques, as developed in
\cite{BT82,Kol98,Ceg98,Zer01,GZ05,BGZ09}.

\begin{proof}
Recall that the Monge-Amp\`{e}re capacity has been introduced by Bedford and Taylor in \cite{BT82}. By definition
the capacity of a compact subset $K \subset \Omega$ is
\begin{equation*}
  Cap(K):=\sup \set{\int_K (dd^{c} u)^{n} ; \; u \text{ plurisubharmonic in } \Omega \text{ with } 0 \leq u \leq 1}.
\end{equation*}

We will use the following useful inequalities. For any $\gamma<2$
there exists $C_{\gamma}>0$ such that for all $K \subset \Omega$,
\begin{equation}
  \mu(K) \leq C_{\gamma} \exp \left[ -\frac{\gamma}{Cap(K)^{1/n}} \right]
\end{equation}
(see e.g. \cite{Zer01}). For all smooth plurisubharmonic functions $\f$ in $\Omega$
with zero boundary values, for all $t>0$,
\begin{equation*}
  Cap(\f<-t) \leq \frac{(n+1) \abs{\mathcal{E}(\f)}}{t^{n+1}},
\end{equation*}
where
\begin{equation*}
  \mathcal{E}(\f):=\frac{1}{n+1}\int_{\Omega} \f (dd^{c} \f)^{n}.
\end{equation*}
For the latter inequality, we refer the reader to Lemma 2.2 in~\cite{ACKPZ09}. We infer
\begin{equation*}
  \int_{\Omega} e^{-\f} d \mu=-1+\int_{0}^{+\infty} e^{t} \mu(\f<-t) dt \leq C \int_{0}^{+\infty} \exp(t-\lambda t^{1+1/n}) dt,
\end{equation*}
where
\begin{equation*}
  \lambda:=\frac{\gamma}{(n+1)^{1/n} \abs{\mathcal{E}(\f)}^{1/n}}.
\end{equation*}

We let the reader check that the function $h(t)=t-\lambda t^{1+1/n}$ attains its maximum value at
point $t_c=\lambda^{-n} (1+1/n)^{-n}$. Moreover $h(t) \leq -t$ for $t \geq 4^{n} t_c$. This shows that
\begin{equation*}
\begin{split}
 \int_{0}^{+\infty} \exp(t-\lambda t^{1+1/n}) dt & \le 4^{n} t_c \exp(h(t_c))+\int_{4^{n} t_c}^{+\infty} \exp(-t) dt \\
  & \le 4^{n} t_c \exp \left( \frac{t_c}{n+1} \right)+1.
\end{split}
\end{equation*}
Using the definition of $\lambda$ and the formula defining $t_c$, we arrive at
\begin{equation*}
  \int_{0}^{+\infty} \exp(t-\lambda t^{1+1/n}) dt \leq c_{n} \abs{\mathcal{E}(\f)} \exp\left( {\beta_{n}'} \abs{\mathcal{E}(\f)} \right)+1,
\end{equation*}
where
\begin{equation*}
  \beta_{n}'=\frac{1}{\gamma^{n} (1+1/n)^{n}}.
\end{equation*}
We can fix e.g. $\gamma=1$ so that $\beta_{n}'<1$ for all $n \geq 1$. Moreover the desired inequality is obtained by choosing
$\beta_{n}$ so that $\beta_{n}'<\beta_{n}<1$ and enlarging the constant $C$.
\end{proof}

\begin{rem} \label{rem:biggervalues}
Note for later use that the same proof yields an inequality
\begin{equation} \label{eq:maxRicci}
  \int_{\Omega} e^{-A \f} d \mu \leq C_A \exp( \beta_A \abs{\mathcal{E}(\f)} ),
\end{equation}
where
\begin{equation*}
  \beta_A:=\frac{A^{n+1}}{\gamma^{n} (1+1/n)^{n}}
\end{equation*}
is smaller than $1$ only if $A=A_{n}$ is not too large. When $n=1$, the critical value is $A=2$.
This is related to a theorem of Bishop as we shall see in section~\ref{subsec:conclusion}.

It follows from the recent work \cite{ACKPZ09} that the optimal exponent $\gamma$ is actually $2n$,
improving the bound $2$ obtained in \cite{Zer01}, hence also enlarging the allowed constant
$A_{n}$ above, when $n>1$.
\end{rem}


\subsection{Properness}\label{subsec:properness}

We let
\begin{equation*}
  \mathcal{E}(\f):=\frac{1}{n+1} \int_{\Omega} \f (dd^{c} \f)^{n}
\end{equation*}
denote the \emph{energy} of a plurisubharmonic function $\f$ and set
\begin{equation*}
  \mathcal{F}(\f):=\mathcal{E}(\f)+\log \left[ \int_{\Omega} e^{-\f} d \mu \right].
\end{equation*}

Recall that the energy functional is a primitive of the complex Monge-Amp\`{e}re operator, namely if
$\p_s$ is a curve of plurisubharmonic functions with zero boundary values, then
\begin{equation*}
  \frac{d\mathcal{E}(\p_s)}{ds}=\int_{\Omega} \dot{\p_s} (dd^{c} \p_s)^{n},
\end{equation*}
as follows from Stokes theorem. A similar computation shows that a function $\f$ solves $(MA)$ if
and only if it is a critical point of the functional $\mathcal{F}$ (in other words
$(MA)$ is the Euler-Lagrange equation for $\mathcal{F}$).

Inspired by techniques from the calculus of variations, it is thus natural to try and maximize the functional
$\mathcal{F}$ so as to build a critical point. This usually requires the functional to be \emph{proper} in order
to be able to restrict to compact subsets of the space of functions involved.
It follows from the Moser-Trudinger inequality (Theorem~\ref{thm:MT}) that the functional $\mathcal{F}$ is
indeed \emph{proper},
in the following strong sense:

\begin{prop} \label{prop:proper}
There exists $a>0, b \in \RR$ such that for all smooth plurisubharmonic function $\p$ in $\Omega$, with zero boundary values,
\begin{equation*}
  \mathcal{F}(\p) \leq a \mathcal{E}(\p)+b.
\end{equation*}
\end{prop}

\begin{proof}
Immediate consequence of Theorem~\ref{thm:MT} with $a=1-\beta_{n}$ and $b=\log C$.
\end{proof}


\subsection{Ricci inverse iteration}

Given $\f \in PSH(\Omega) \cap \mathcal{C}^{\infty}(\bar{\Omega})$ with zero boundary values,
it follows from the work of Cafarelli, Kohn, Nirenberg and Spruck \cite{CKNS85} that there exists a unique
function $\p \in PSH(\Omega) \cap \mathcal{C}^{\infty}(\bar{\Omega})$ with zero boundary values
such that
\begin{equation*}
(dd^{c} \p)^{n}=\frac{e^{-\f} \mu}{\int_{\Omega} e^{-\f} d \mu}
\text{ in } \Omega.
\hskip2cm (*)
\end{equation*}
We let
\begin{equation*}
  \mathcal{T}:=\set{\f \in PSH(\Omega) \cap \mathcal{C}^{\infty}(\bar{\Omega}) \, | \, \f_{|\partial \Omega} =0}
\end{equation*}
denote the space of test functions and
\begin{equation*}
  T: \f \in \mathcal{T} \mapsto \p \in \mathcal{T}
\end{equation*}
denote the operator such that $\p=T(\f)$ is the unique solution of $(*)$.
Observe that solving $(MA)$ is equivalent to finding a fixed point of $T$.

The key to the dynamical construction
of solutions to $(MA)$ lies in the following monotonicity property:

\begin{prop} \label{prop:monotonediscrete}
For all $\f \in \mathcal{T}$,
\begin{equation*}
  \mathcal{F}(T \f) \geq \mathcal{F}(\f)
\end{equation*}
with strict inequality unless $T\f=\f$.
\end{prop}

\begin{proof}
Fix $\f \in \mathcal{T}$ and set $\p:=T \f$. Recall that
\begin{equation*}
\mathcal{F}(\f)=\mathcal{E}(\f)+\log \left[ \int_{\Omega} e^{-\f} d\mu \right]
\end{equation*}
and
\begin{equation*}
  \mathcal{E}(\p)-\mathcal{E}(\f)=\frac{1}{n+1} \sum_{j=0}^{n} \int_{\Omega} (\p-\f) (dd^{c} \p)^j \wedge (dd^{c} \f)^{n-j}.
\end{equation*}

It follows from Stokes theorem that for all $j$,
\begin{multline*}
  \int (\p-\f) (dd^{c} \p)^j \wedge (dd^{c} \f)^{n-j} =\int (\p-\f) (dd^{c} \p)^{n} \\
  + \int d(\p-\f) \wedge d^{c} (\p-\f) \wedge S,
\end{multline*}
where $S$ is a positive closed form of bidegree $(n-1,n-1)$. Thus
\begin{equation*}
  \mathcal{E}(\p)-\mathcal{E}(\f) \geq \frac{1}{n+1} \int_{\Omega} (\p-\f) (dd^{c} \p)^{n}.
\end{equation*}

We now set
\begin{equation*}
  \tilde{\f}:=\f+\log [ \int e^{-\f} d\mu ], \qquad \tilde{\p}:=\p+\log [ \int e^{-\p} d\mu ],
\end{equation*}
and
\begin{equation*}
  \mu_{\f}:=e^{-\tilde{\f}} \mu , \qquad \mu_{\p}:=e^{-\tilde{\p}} \mu.
\end{equation*}
Note that the latter are probability measures in $\Omega$ with $(dd^{c} \p)^{n}=\mu_{\f}$.

It follows from the definition of $\mathcal{F}$ and our last inequality that
\begin{equation*}
  \mathcal{F}(\p)-\mathcal{F}(\f) \geq \int_{\Omega} (\tilde{\p}-\tilde{\f}) d \mu_{\f} = \int_{\Omega} F \log F \, d \mu_{\p},
\end{equation*}
where $F=e^{\tilde{\p}-\tilde{\f}}$, hence the latter quantity denotes the relative entropy of the
probability measures $\mu_{\f},\mu_{\p}$. It follows from the convexity of $-\log$ that
\begin{equation*}
  \int_{\Omega} - \log [ F^{-1}] \, F d \mu_{\p} \geq -\log \left[ \int_{\Omega} F^{-1} F d \mu_{\p} \right]=0,
\end{equation*}
with strict inequality unless $F=1$ almost everywhere, i.e. $\tilde{\f}=\tilde{\p}$.

Observe finally that since $\p$ and $\f$ both have zero boundary values, the equality
$\tilde{\f}=\tilde{\p}$ can only occur when $\f \equiv \p$, i.e. when
$\f=T \f$ is a fixed point of $T$, as claimed.
\end{proof}

We infer that the energies $\mathcal{E}(\f_{j})$ of the solutions $\f_{j}$ of $(MA)_{j-1}$ are uniformly bounded:

\begin{cor}
The sequence $(\mathcal{F}(T^j \f_{0}))_{j}$ is bounded, hence so is
$(\mathcal{E}(T^j \f_{0})_{j}$.
\end{cor}

\begin{proof}
Fix $\f_{0} \in \mathcal{T}$ (for example $\f_{0}=\rho$) and set $\f_{j}=T^j \f_{0}$.
Observe that $\mathcal{E}(\f_{j}) \leq 0$ since
$\f_{j} \leq 0$, hence it suffices to establish a bound from below.
 The previous proposition insures that the sequence $\mathcal{F}(T^j \f_{0}))_{j}$
is increasing. It follows from Proposition~\ref{prop:proper} that
\begin{equation*}
\mathcal{F}(\f_{0}) \leq \mathcal{F}(T^j \f_{0}) \leq a \mathcal{E}(T^j \f_{0})+b \leq b
\end{equation*}
so that the energies $\mathcal{E}(T^j \f_{0})$ are uniformly bounded.
\end{proof}


\section{Higher order estimates}
\label{sec:higher_order_estimates}


\subsection{Uniform a priori estimates}

Recall that $\f_{j}$ is a smooth plurisubharmonic solution of $(MA)_{j-1}$. Its Monge-Amp\`{e}re measure thus satisfies
\begin{equation*}
  (dd^{c} \f_{j})^{n}=f_{j} \mu, \text{ with } f_{j}=\frac{e^{-\f_{j-1}} \mu}{\int_{\Omega} e^{-\f_{j-1}} d \mu}.
\end{equation*}
It follows from the previous section that the $\f_{j}'$s have uniformly bounded energy.
Thus they form a relatively compact family (for the $L^1$-topology) in the class
$\mathcal{E}^1(\Omega)$ of plurisubharmonic functions with finite energy (see \cite{BGZ09}). When
the complex dimension is $n=1$, the latter is the class of negative plurisubharmonic functions with
zero boundary values and whose gradient is in $L^{2}$; since (normalized) plurisubharmonic functions
are uniformly $L^{2}$, the family $(\f_{j})$ is thus included in a finite ball of the Sobolev space $W^{1,2}$.
In higher dimension, the class $\mathcal{E}^1(\Omega)$ is a convenient substitute for the Sobolev spaces,
we refer the reader to \cite{BGZ09} for more details.

We simply recall here that functions in $\mathcal{E}^1(\Omega)$ have zero Lelong numbers.
For such a function $\p$, Skoda's integrability theorem \cite{Sko72} ensures that $e^{-\p}$ is
in $L^q$ for all $q>1$. Since the family $(\f_{j})$ is moreover relatively compact,
Skoda's uniform integrability theorem
\cite{Zer01} insures that the densities $f_{j}$'s satisfy
\begin{equation*}
  \int_{\Omega} f_{j}^{2} d \mu \leq C
\end{equation*}
for some uniform constant $C>0$.

Recall now the following fundamental result due to Kolodziej \cite{Kol98}:
if $\p$ is a smooth plurisubharmonic function in $\Omega$ with zero boundary
values and such that
\begin{equation*}
  (dd^{c} \p)^{n}=f d\mu
\end{equation*}
where $f \in L^{2}(\mu)$, then
\begin{equation*}
  \norm{\p}_{L^{\infty}(\Omega)} \leq C_f,
\end{equation*}
where the constant $C_f$ only depends on $\Omega$ and $\norm{f}_{L^{2}}$.
Applying this to $\p=\f_{j}$ yields:

\begin{lem} \label{lem:C0}
For all $j \in \NN$,
\begin{equation} \label{eq:uniform}
  -C_{0} \leq \f_{j} \leq 0
\end{equation}
for some uniform constant $C_{0}>0$.
\end{lem}


\subsection{$\mathcal{C}^{2}$-a priori estimates}
\label{subsec:apriori_estimates}

The goal of this section is to establish the following a priori estimates on the Laplacian
of the solutions to $(MA)_{j-1}$.

\begin{thm} \label{thm:C2}
There exists $C>0$ such that for all $j \in \NN$,
\begin{equation*}
  \sup_{\bar{\Omega}} \abs{\Delta \f_{j}} \leq C.
\end{equation*}
\end{thm}

These estimates are ``almost'' contained in \cite{CKNS85}, however hypothesis (1.3) on p. 213 is
not satisfied, hence neither \cite[Theorem 1.1]{CKNS85} nor \cite[Theorem 1.2]{CKNS85} can be applied
to our situation.

We nevertheless follow their proof as organized by S. Boucksom \cite{Bou11},
explaining some of the necessary adjustments. It will be a consequence of the following series of lemmas.

\begin{lem}\label{lem:C1}
There exists $C_{1}>0$ such that
\begin{equation*}
  \sup_{\partial \Omega} \abs{\nabla \f_{j}} \leq C_{1}.
\end{equation*}
\end{lem}

\begin{proof}
It follows from the order zero uniform estimates (\ref{eq:uniform}) that
\begin{equation*}
  (dd^{c} \f_{j})^{n} \leq e^{C_{0}} \mu \text{ in } \Omega.
\end{equation*}
Let $u$ denote the unique smooth plurisubharmonic function in $\bar{\Omega}$ such that
\begin{equation*}
(dd^{c} u)^{n}=e^{C_{0}} \mu
\text{ in } \Omega
\text{ with }
u_{|\partial \Omega} \equiv 0.
\end{equation*}
The latter exists by \cite[Theorem 1.1]{CKNS85}. It follows from the comparison principle that
\begin{equation*}
  u \leq \f_{j} \leq 0 \text{ in } \Omega.
\end{equation*}
This yields the desired control of $\nabla \f_{j}$ on $\partial \Omega$.
\end{proof}

\begin{lem}
There exists $C_2>0$ such that
\begin{equation*}
\sup_{\Omega} \abs{\Delta \f_{j}} \leq C_2 (1+\sup_{\partial \Omega} \abs{\Delta \f_{j}} ).
\end{equation*}
\end{lem}

\begin{proof}
We let $\Delta_{j}$ denote the Laplace operator with respect to the K\"{a}hler form $\omega_{j}=dd^{c} \f_{j}$, while
$\Delta$ denotes the euclidean Laplace operator. We claim that for all $j \geq 1$,
\begin{equation} \label{eq:Yau}
\Delta_{j} \left\{ \log \Delta \f_{j} +\f_{j-1} \right\} \geq 0.
\end{equation}

Assuming this for the moment we show how to derive the desired control on $\Delta \f_{j}$.
Let $z_{j} \in \bar{\Omega}$ be a point which realizes the maximum of the function
\begin{equation*}
  h_{j}:=\f_{j}+\f_{j-1}+\log \Delta \f_{j}.
\end{equation*}
It follows from (\ref{eq:Yau}) that $z_{j} \in \partial \Omega$, otherwise $\Delta_{j} h_{j}(z_{j}) \leq 0$
contradicting
\begin{equation*}
  \Delta_{j} h_{j} \geq \Delta_{j} \f_{j} >0.
\end{equation*}
We infer from Lemma~\ref{lem:C0} that for all $w \in \Omega$,
\begin{equation*}
  \log \Delta \f_{j} (w) \leq 2 C_{0} + h_{j}(z_{j}) \leq 2C_{0}+\log \sup_{\partial \Omega} \Delta \f_{j},
\end{equation*}
which yields the desired upper bound.

It remains to establish $(\ref{eq:Yau})$. We shall need the following local differential inequality which goes back to the works of Aubin and Yau: if $\omega$ is an arbitrary K\"{a}hler form and $\beta=dd^{c} \norm{z}^{2}$
denotes the euclidean K\"{a}hler form, then
\begin{equation} \label{eq:Aubin}
 \Delta_{\omega} \log \tr_{\beta}(\omega) \geq -\frac{\tr_{\beta} (Ric \omega)}{\tr_{\beta}(\omega)}.
\end{equation}

We apply this inequality to $\omega=\omega_{j}=dd^{c} \f_{j}$. Observe that $Ric(\omega_{j})=\omega_{j-1}$
since
\begin{equation*}
  (dd^{c} \f_{j})^{n}=e^{-\f_{j-1}} e^{c_{j}} dV.
\end{equation*}
Observe that
\begin{equation*}
  \frac{\tr_{\beta}(\omega_{j-1})}{\tr_{\beta}(\omega_{j})} =\frac{\Delta_{\beta}(\f_{j-1})}{\tr_{\beta}(\omega_{j})}
\leq \Delta_{j} (\f_{j-1}).
\end{equation*}
Combined with $(\ref{eq:Aubin})$, this yields
\begin{equation*}
  \Delta_{j} \log \tr_{\beta}(\omega_{j})
\geq -\Delta_{j} (\f_{j-1}),
\end{equation*}
whence $(\ref{eq:Yau})$.
\end{proof}

\begin{lem}
There exists $C_3>0$ such that
\begin{equation*}
  \sup_{\partial \Omega} \abs{D^{2} \f_{j}} \leq C_3 (1+\sup_{\Omega} \abs{\nabla \f_{j}}^{2} ).
\end{equation*}
\end{lem}

\begin{proof}
This follows from a long series of estimates which are the same as those of \cite{CKNS85}, up
to minor modifications. We only sketch these out, following the proof of \cite[Lemma 7.17]{Bou11}.
To fit in with the notations of \cite{Bou11}, we set $\p=\f_{j}-\rho$ and $\eta=dd^{c} \rho$
so that $\p$
is a $\eta$-psh function (still) with zero boundary values on $\partial \Omega$ such that
\begin{equation*}
  (\eta+dd^{c} \p)^{n}=e^{-\p} e^{F} \eta^{n}
\end{equation*}
where $F$ is some smooth density. Our problem is thus equivalent to showing an a priori estimate
\begin{equation*}
  \sup_{\partial \Omega} \abs{D^{2} \p} \leq C_3 (1+\sup_{\Omega} \abs{\nabla \p}^{2} ),
\end{equation*}
where $C_3$ is under control.

Fix $p\in \Omega$. It is classical that one can choose complex coordinates $(z_{j})_{1 \leq j \leq n}$
so that $p=0$ and
\begin{equation*}
  \rho=-x_{n}+\Re \left( \sum_{j,k=1}^{n} a_{jk} z_{j}\zbar_{k} \right) +O(\abs{z}^3)
\end{equation*}
where $z_{j}=x_{j}+iy_{j}$. We set for convenience
\begin{equation*}
  t_{1}=x_{1}, \quad t_2=x_{1}, \quad \ldots, \quad t_{2n-1}=y_{n},\quad t_{2n}=x_{2n}.
\end{equation*}
Let $(D_{j})$ be the dual basis of $dt_{1},\ldots,dt_{2n-1},-d\rho$ so that for $j<2n$,
\begin{equation*}
   D_{j} = \frac{\partial}{\partial t_{j}} - \frac{\rho_{t_{j}}}{\rho_{x_{n}}} \frac{\partial}{\partial x_{n}} \; \text{ and } \; D_{2n} = -\frac{1}{\rho_{x_{n}}} \frac{\partial}{\partial x_{n}}
\end{equation*}

\smallskip

\noindent {\it Step 0: bounding the tangent-tangent derivatives.}
Observe that the $D_{j}$'s commute and are tangent to $\partial \Omega$ for $j<2n$, we thus
have a trivial control on the tangent-tangent derivatives at $p=0$,
\begin{equation*}
 D_{i}D_{j}\p(0)=0, \; \text{ for } 1 \leq i,j <2n.
\end{equation*}

\smallskip

\noindent {\it Step 1: bounding the normal-tangent derivatives.}
Set $K=\sup_{\partial \Omega} \abs{\nabla \p}$. We claim that for all $1 \leq i<2n$
\begin{equation*}
 \abs{D_{i}D_{2n} \p(0)} \leq C (1+K),
\end{equation*}
for some uniform constant $C>0$.

Let $h$ be the smooth function in $\Omega$ with zero boundary values such that
\begin{equation*}
 \Delta_{\eta}h:=n \frac{dd^{c} h \wedge \eta^{n-1}}{\eta^{n}}=-n
\text{ in } \Omega.
\end{equation*}
The proof requires the construction of a barrier $b=\p+\eps h-\mu \rho^{2}$ such that
\begin{equation*}
 0 \leq b \; \text{ and }
\Delta_{\p}b :=n \frac{dd^{c} b \wedge (\eta+dd^{c} \p)^{n-1}}{(\eta+dd^{c} \p)^{n}}
\leq -\frac{1}{2} \tr_{\p}(\eta) \; \text{ in } B,
\end{equation*}
where $B$ is a half ball centered at $p=0$ of positive radius and $\eps,\mu>0$ are under control.
This can be done exactly as in \cite[Lemma 7.17, Step 1]{Bou11}, as the only information needed is that
$(\eta+dd^{c} \p)^{n}$ is uniformly bounded from above by $C \eta^{n}$, which follows here from our
${\mathcal C}^0$-estimate.

One then shows the existence of uniform constants $\mu_{1},\mu_2>0$ such that the functions
$v_{\pm}:=K(\mu_{1}+\mu_2\abs{z}^{2}) \pm D_{j} \p$ both satisfy
\begin{equation*}
0 \leq v_{\pm} \; \text{ on } B
\; \text{ and } \;
\Delta_{\p} v_{\pm} \leq 0
\; \text{ in } B.
\end{equation*}

It follows then from the maximum principle that $v_{\pm} \geq 0$ in $B$ so that
$D_{2n} v_{\pm}(0) \geq 0$ since $v_{\pm}(0)=0$. Thus
\begin{equation*}
 \abs{D_{2nj} \p(0)} \leq C K (1+D_{2n} b(0) ) \leq C' (1+K),
\end{equation*}
as claimed.

\smallskip

\noindent {\it Step 2: bounding the normal-normal derivatives.}
This is somehow the most delicate estimate.
Set again $K=\sup_{\partial \Omega} \abs{\nabla \p}$.
We want to show that
$\abs{D^{2}_{2n} \p(0)} \leq C (1+K^{2})$ for some uniform constant $C>0$.
Using previous estimates on $D_{i}D_{j}\p(0)$, it suffices to show that
\begin{equation*}
    \abs{\p_{z_{n}\zbar_{n}}(0)} \leq C (1+K^{2}).
\end{equation*}

Recall that
\begin{equation*}
 \det \left( \rho_{z_{i} \zbar_{j}}(0)+ \p_{z_ i\zbar_{j}} (0) \right)_{1 \leq i,j \leq n}=e^{-\p(0)+F(0)}
\end{equation*}
is bounded from above, and for $i<n$,
\begin{equation*}
  \abs{\p_{z_{i}\zbar_{n}}(0)} \leq C (1+K).
\end{equation*}
Expanding the determinant with respect to the last row
thus yields the expected upper bound, provided we can bound from below
the $(n-1,n-1)$-minor
\begin{equation*}
  \det \left( \rho_{z_{i} \zbar_{j}}(0)+ \p_{z_ i\zbar_{j}} (0) \right)_{1 \leq i,j \leq n-1}.
\end{equation*}
A (by now) classical barrier argument
shows that $dd^{c} \f=\eta+dd^{c} \p$ is uniformly bounded from below by $\eps \eta$ on the complex
tangent space to $\partial \Omega$
(see \cite[Lemma 7.16]{Bou11} which can be used since $\f_{j}$ is uniformly bounded).
\end{proof}

\begin{lem}
There exists $C_4>0$ such that
\begin{equation*}
\sup_{\Omega} \abs{\nabla \f_{j}} \leq C_4.
\end{equation*}
\end{lem}

\begin{proof}
It follows from previous estimates that
\begin{equation*}
\sup_{\Omega} \Delta \f_{j} \leq C \left( 1+ \sup_{\Omega} \abs{\nabla \f_{j}}^{2} \right).
\end{equation*}

Assume that $\sup_{\Omega} \abs{\nabla \f_{j}}$ is unbounded. Up to extracting and relabelling, this means that
\begin{equation*}
  M_{j}: = \abs{\nabla \f_{j} (x_{j})} = \sup_{\Omega} \abs{\nabla \f_{j}} \rightarrow + \infty
\end{equation*}
where $x_{j} \in \bar{\Omega}$ converges to $a \in \bar{\Omega}$. We set
\begin{equation*}
  \p_{j}(z) := \f_{j}(x_{j}+M_{j}^{-1}z).
\end{equation*}

This is a sequence of uniformly bounded plurisubharmonic functions which are well defined (at least) in a half ball $B$ around zero and satisfy
\begin{equation*}
  \abs{\nabla \p_{j}(0)}=1 \; \; \text{ and } \sup_B \Delta \p_{j} \leq C.
\end{equation*}
We infer that the sequence $(\p_{j})$ is relatively compact in $\mathcal{C}^1$, hence we can assume
that (up to relabeling) $\p_{j} \rightarrow \p \in {\mathcal C}^1(B)$ where $\p$ is plurisubharmonic and satisfies $\nabla \p(0)=1$.

If $a \in \partial \Omega$, it follows from the proof of Lemma~\ref{lem:C1} that $\p \equiv 0$,
contradicting $\nabla \p(0)=1$. Therefore $a \in \Omega$, so we can actually assume that $B$ is a ball of
arbitrary size, hence $\p$ can be extended as a plurisubharmonic function on the whole of $\CC^{n}$.
Since $\f_{j}$ is uniformly bounded, so are $\p_{j}$ and $\p$. Thus $\p$ has to be constant, contradicting
$\nabla \p(0)=1$.
\end{proof}


\subsection{Evans-Krylov theory}
\label{subsec:evans_krylov}

It follows from Schauder's theory for linear elliptic equations with variable coefficients
that it suffices to obtain a priori estimates
\begin{equation} \label{eq:2}
  \norm{\f_{j}}_{2,\alpha} \leq C
\end{equation}
for some positive exponent $\alpha>0$, in order to obtain a priori estimates
\begin{equation} \label{eq:k+2}
  \norm{\f_{j}}_{k+2,\alpha} \leq C_k
\end{equation}
at all orders $k \in \NN$. Here
\begin{equation*}
   \norm{h}_{k,\alpha}:=\sum_{j=0}^k \sup_{\Omega} \abs{D^j h}+\sup_{z,w \in \Omega, z \neq w} \frac{\abs{D^kh(z)-D^kh(w)}}{\abs{z-w}^{\alpha}}
\end{equation*}
denotes the norm associated to the H\"older space of functions $h$ which are $k$-times differentiable on $\bar{\Omega}$
with $k^{th}$-derivative H\"older-continuous of exponent $\alpha>0$.

Moreover the Evans-Krylov theory (as simplified by Trudinger) can be adapted to the case
of complex Monge-Amp\`{e}re equations,
showing that the a priori estimates (\ref{eq:2}) follow directly from Theorem~\ref{thm:C2}.
We refer the reader to \cite{Blo05} for a detailed presentation of this material.


\subsection{Conclusion}
\label{subsec:conclusion}

It follows from the previous sections that the sequence $(\f_{j})$ is relatively compact in
$\mathcal{C}^{\infty}(\bar{\Omega})$.
We let $\mathcal{K}$ denote the set of its cluster values. We infer from Proposition~\ref{prop:monotonediscrete}
that the functional $\mathcal{F}$ is constant on $\mathcal{K}$: for
all $\p \in \mathcal{K}$,
\begin{equation*}
  \mathcal{F}(\p)=\lim_{j \rightarrow +\infty} \nearrow \mathcal{F}(T^j \f_{0}).
\end{equation*}

Now $\mathcal{K}$ is clearly $T$-invariant, hence
$\mathcal{F}(T\p)=\mathcal{F}(\p)$ for all $\p \in \mathcal{K}$.
Thus Proposition~\ref{prop:monotonediscrete} again insures that
$T\p=\p$, i.e. $\p$ is a solution of $(MA)$.

As explained earlier, this is equivalent to saying that there exists
a K\"{a}hler-Einstein metric $\omega=dd^{c} \f$ with $Ric(\omega)=\pi\omega$ and
prescribed values on the boundary of $\Omega$, hence we have solved our geometrical problem.


\section{Uniqueness}
\label{sec:variational_characterization}

Recall that $(MA)$ is the Euler-Lagrange equation of the functional
\begin{equation*}
\mathcal{F}(\f):=\mathcal{E}(\f)+\log \left[ \int_{\Omega} e^{-\f} d \mu \right].
\end{equation*}

If a smooth strictly plurisubharmonic function $\f$
with zero boundary values maximizes ${\mathcal F}$, then it is a critical point of ${\mathcal F}$
hence $\f$ is a solution of $(MA)$. Indeed for any smooth function $v$
\begin{equation*}
  \frac{d}{dt}{\mathcal F}(\f+tv)_{|t=0}=\int_{\Omega} v (dd^{c} \f)^{n} - \frac{\int_{\Omega} v e^{-\f} d\mu}{\int_{\Omega} e^{-\f} d\mu}=0,
\end{equation*}
thus $(dd^{c} \f)^{n} =e^{-\f} \mu /\left( \int_{\Omega} e^{-\f} d \mu \right)$.

Our purpose here is to show that the converse holds true when
$\Omega$ satisfies an additional symmetry property.


\subsection{Continuous geodesics}
\label{subsec:continuous_geodesics}

In the setting of compact K\"{a}hler manifolds, Mabuchi \cite{Mab87}, Semmes \cite{Sem92} and Donaldson \cite{Don99}
have shown that the set of all K\"{a}hler metrics in a fixed cohomology class has the structure of an infinite Riemannian
manifold with non negative curvature. The notion of geodesic joining two K\"{a}hler metrics plays an important role there
and we refer the reader to \cite{Che00} for more information on this.

Our purpose here is to consider similar objects for pseudoconvex domains in order to study the uniqueness
of solutions to $(MA)$. Let $A$ denote the annulus $A=\{ \zeta \in \CC \, / \, 1 < \abs{\zeta} < e \}$ and fix two functions
$\phi_{0},\phi_{1}$ which are plurisubharmonic in $\Omega$, continuous up to the boundary, with zero boundary values.
We let $\mathcal{G}$ denote the set of all plurisubharmonic functions $\Psi$ on $\Omega \times A$
which are continuous on $\bar{\Omega} \times \bar{A}$ and such that
\begin{equation*}
  \Psi_{|\partial \Omega \times A} \equiv 0 \quad \text{and} \quad \Psi_{| \Omega \times \partial A} \leq \phi,
\end{equation*}
where $\phi(z,\zeta)=\phi_{0}(z)$ for $\abs{\zeta}=1$ and $\phi(z,\zeta)=\phi_{1}(z)$ for $\abs{\zeta}=e$. We set
\begin{equation*}
  \Phi(z,\zeta):=\sup \left\{ \Psi(z,\zeta) \, / \, \Psi \in \mathcal{G} \right\}.
\end{equation*}

\begin{prop}\label{prop:geodesic}
The function $\Phi$ is plurisubharmonic in $\Omega \times A$, continuous on $\bar{\Omega} \times \bar{A}$ and satisfies
\begin{itemize}
 \item [(i)] $\Phi(z,e^{i\theta} \zeta)=\Phi(z,\zeta)$ for all $(z,\zeta,\theta) \in \Omega \times A \times \RR$;
 \item [(ii)] $\Phi(z,1)=\phi_{0}(z)$ and $\Phi(z,e)=\phi_{1}(z)$ for all $z \in \Omega$;
 \item [(iii)] $(dd_{z,\zeta}^{c} \Phi)^{n+1} \equiv 0$ in $\Omega \times A$.
\end{itemize}
\end{prop}

\begin{proof}
The invariance by rotations (i) follows from the corresponding invariance property of the family $\mathcal{G}$. The continuity and boundary properties (ii) follow standard arguments which go back to Bremermann \cite{Bre59} and Walsh \cite{Wal69}.

The maximality property (iii) is a consequence of Bedford-Taylor's solution to the homogeneous complex Monge-Amp\`{e}re equation on balls, through a balayage procedure: by Choquet's lemma, the sup can be achieved along an increasing sequence which is maximal on an arbitrary ball $B \subset \Omega \times A$, one then concludes by using the continuity property of the complex Monge-Amp\`{e}re operator along increasing sequences \cite{BT82}.
\end{proof}

\begin{defi}
Set $\Phi_{t}(z)=\Phi(z,e^{t})$. The continuous family $(\Phi_{t})_{0 \leq t \leq 1}$ is called the geodesic
joining $\phi_{0}$ to $\phi_{1}$.
\end{defi}

Recall that
\begin{equation*}
  \mathcal{E}(\f):=\frac{1}{n+1} \int_{\Omega} \f (dd^{c} \f)^{n}
\end{equation*}
denotes the \emph{energy} of a plurisubharmonic function $\f$.

\begin{lem}
Let $(\Phi_{t})_{0 \leq t \leq 1}$ be a continuous geodesic. Then $t \mapsto \mathcal{E}(\Phi_{t})$ is affine.
\end{lem}

\begin{proof}
We let the reader verify that if $(z,\zeta) \mapsto \Phi(z,\zeta)$ is a continuous plurisubharmonic function in $\Omega \times A$, then
\begin{equation*}
  dd^{c}_{\zeta} \, \mathcal{E} \circ \Phi=\pi_*\left( (dd^{c}_{z,\zeta} \Phi)^{n+1} \right),
\end{equation*}
where $\pi:\Omega \times A \rightarrow A$ denotes the projection onto the second factor.

It thus follows from Proposition~\ref{prop:geodesic} that $\zeta \in A \mapsto \mathcal{E} \circ \Phi(\zeta) \in \RR$ is harmonic in $\zeta$.
The same proposition insures that it is also invariant by rotation, hence it is affine in $t=\log\abs{\zeta}$.
\end{proof}


\subsection{Variational characterization}

We now make an additional hypothesis of $S^1$-invariance in order to  use an important result by Berndtsson \cite{Ber06}. Namely we assume here below that $\Omega$ is {\it circled}, i.e.
\begin{center}
    $\Omega$ \emph{contains the origin and is invariant under the rotations} $z \mapsto e^{i\theta} z$,
\end{center}
and
\begin{center}
$\phi_0,\phi_1$ are $S^1$-invariant, i.e. $\phi_i(e^{i \theta} z)=\phi_i(z)$.
\end{center}
Under this assumption, it follows from \cite[Theorem 1.2]{Ber06} that
\begin{equation*}
  t \mapsto -\log \left( \int_{\Omega} e^{-\Phi_{t}} d\mu \right)
\end{equation*}
is a convex function of $t$ if $(\Phi_{t})$ is a continuous geodesic.

\begin{prop} \label{prop:variational}
Assume $\Omega$ is circled and let $\f$ be a $S^1$-invariant solution of $(MA)$. Then
\begin{equation*}
    \mathcal{F}(\f) \geq \mathcal{F}(\psi),
\end{equation*}
for all $S^1$-invariant plurisubharmonic functions $\psi$ in $\Omega$ which are continuous up to the boundary, with zero boundary values.
\end{prop}

\begin{proof}
Let $(\Phi_{t})_{0 \leq t \leq 1}$ denote the \emph{geodesic} joining $\phi_{0}:=\f$ to $\phi_{1}:=\p$. It follows from the above mentioned work of Berndtsson \cite{Ber06} that
\begin{equation*}
  t \mapsto -\log \left( \int e^{-\Phi_{t}} d \mu \right)
\end{equation*}
is convex, while we have just observed that
\begin{equation*}
  t \mapsto \mathcal{E}(\Phi_{t})
\end{equation*}
is affine, thus
\begin{equation*}
  t \mapsto \mathcal{F}(\Phi_{t}) \text{ is concave.}
\end{equation*}

It therefore suffices to show that the derivative of $\mathcal{F}(\Phi_{t})$ at $t=0$ is non positive to conclude that
$\mathcal{F}(\f)=\mathcal{F}(\Phi_{0}) \geq \mathcal{F}(\Phi_{t})$ for all $t$, in particular
at $t=1$ where it yields $\mathcal{F}(\f) \geq \mathcal{F}(\p)$.
When $t \mapsto \Phi_{t}$ is smooth, a direct computation yields, for $t=0$,
\begin{equation*}
  \frac{d}{dt}(\mathcal{F}(\Phi_{t}))=\int_{\Omega} \dot{\Phi_{t}} \left[ (dd^{c} \Phi_{t})^{n}-e^{-\Phi_{t}}\mu/(\int e^{-\Phi_{t}}d \mu) \right]=0
\end{equation*}
since $\Phi_{0}=\f$ is a solution of $(MA)$. For the general case, one can argue as in the proof of
Theorem 6.6 in \cite{BBGZ09}.
\end{proof}

\begin{cor} \label{cor:variational}
A smooth $S^1$-invariant plurisubharmonic function $\f:\bar{\Omega} \rightarrow \RR$
with zero boundary values is a solution of $(MA)$, i.e. satisfies
\begin{equation*}
  (dd^{c} \f)^{n}=\frac{e^{-\f} \mu}{\int_{\Omega} e^{-\f} d \mu} \quad \text{in} \quad \Omega
\end{equation*}
if and only if it maximizes the functional $\mathcal{F}$.
\end{cor}


\subsection{Uniqueness of solutions}

The purpose of this section is to establish a uniqueness result for $(MA)$. Recall that if $\varphi$ is a solution of $(MA)$, we say that $\Omega$ is strictly $\varphi$-convex if $\Omega$ is strictly convex for the metric $dd^{c}\varphi$.

\begin{thm}
Assume that $\Omega$ is circled and strictly $\f$-convex, where $\f$ is a
$S^1$-invariant solution of $(MA)$.
Then $\f$ is the only $S^1$-invariant solution to $(MA)$.
\end{thm}

\begin{proof}
Assume we are given $\f,\p$ two $S^1$-invariant solutions of $(MA)$. Let
 $(\Phi_{t})_{0 \leq t \leq 1}$ denote the continuous geodesic joining $\phi_{0}=\f$ to $\phi_{1}=\p$. Since the functional
$\mathcal{F}$ is concave along this geodesic and attains its maximum both at $\phi_{0}$ and $\phi_{1}$,
it is actually constant, hence each $\Phi_{t}$ is a $S^1$-invariant solution to $(MA)$ by Corollary~\ref{cor:variational}, so that
\begin{equation*}
  (dd^{c} \Phi_{t})^{n}=\frac{e^{-\Phi_{t}} \mu}{\int_{\Omega} e^{-\Phi_{t}} d\mu} \; \; \text{ in } \; \Omega.
\end{equation*}

Assume that the mapping $(z,t) \in \Omega \times A \mapsto \Phi_{t}(z) \in \RR$ is smooth.
Taking derivatives with respect to $t$, we infer
\begin{equation*}
   n \, dd^{c} \dot{\Phi_{t}} \wedge (dd^{c} \Phi_{t})^{n-1}=\left[-\dot{\Phi_{t}} +\int_{\Omega} \dot{\Phi_{t}} (dd^{c} \Phi_{t})^{n} \right] (dd^{c} \Phi_{t})^{n},
\end{equation*}
so that $1$ is an eigenvalue with eigenvector $\dot{\Phi_{t}}-\int_{\Omega} \dot{\Phi_{t}} (dd^{c} \Phi_{t})^{n} $ for the Laplacian $\Delta_{t}$ associated to the K\"{a}hler form $dd^{c} \Phi_{t}$. Without the regularity assumption,
we can take derivatives in the sense of distributions to insure that at $t=0$,
\begin{equation*}
   n \, dd^{c} \dot{\Phi_{0}} \wedge (dd^{c} \Phi_{0})^{n-1}=\left[-\dot{\Phi_{0}} + \int_{\Omega} \dot{\Phi_{0}} (dd^{c} \Phi_{0})^{n} \right] (dd^{c} \Phi_{0})^{n},
\end{equation*}
as in the proof of Theorem 6.8 in \cite{BBGZ09}. Note that $\Phi_{0}=\f$ is smooth. In particular $\dot{\Phi_{0}}$ is solution of
\begin{equation}\label{equ:elliptic_equation}
  -\triangle\psi = \psi - c(\psi) \; \text{ in } \; \Omega \quad \text{with} \quad {\psi}_{|\partial \Omega}=0,
\end{equation}
where
\begin{equation*}
  c(\psi) = \int_\Omega \psi (dd^{c}\varphi )^{n}.
\end{equation*}

We are going to show that any solution of equation \eqref{equ:elliptic_equation} has to vanish identically if $\Omega$ is strictly $\varphi$-convex. Namely, assume first that $c(\psi)\geq 0$. Write $\psi = \psi^{+} - \psi^{-}$, where $\psi^{+}=\max{\{\psi ,0\}}$ and $\psi^{-}=\max{\{-\psi , 0\}}.$ Multiplying equation \eqref{equ:elliptic_equation} by $\psi^{+}$ and integrating by parts, we get
\begin{equation*}
\begin{split}
    \int_{\Omega} \abs{d\psi^{+}}^{2} (dd^{c}\varphi )^{n} & = \int_{\Omega} (\psi^{+} )^{2} (dd^{c}\varphi )^{n} - c(\psi)\int_\Omega \psi^{+} (dd^{c}\varphi )^{n} \\
     & \leq \int_\Omega (\psi^{+})^{2} (dd^{c}\varphi )^{n}.
 \end{split}
\end{equation*}

By the variational characterization of the first eigenvalue of the Laplacian, if $\psi^{+}$ doesn't vanish identically, then the last inequality means that the first eigenvalue of $\Delta$ with Dirichlet boundary condition is at most $1$. However, by \cite{GKY11}[Corollary 1.2], we know that this eigenvalue is strictly bigger than $1$ because of the strict convexity condition\footnote{Due to our normalization convention for $d^{c}$, there is a $\pi$ factor difference between the definition of $\Delta$ in our present work and the one in \cite{GKY11}.}. This shows that $\psi^{+}=0$ and therefore $\psi=0$ because $c(\psi)\geq 0$. If $c(\psi)\leq 0$, the reasoning is similar and $\psi=0$ as well.

As a conclusion, we see that $\dot{\Phi_{0}} = 0$ on $\Omega$. Therefore, since the energy
\begin{equation*}
  t \mapsto \mathcal{E}(\Phi_{t})
\end{equation*}
is affine along the geodesic, and its derivative at $t=0$ vanishes, it is constant on the interval $[0,1]$. Now, along the geodesic, the derivative of $\mathcal{F}$ vanishes and since
\begin{equation*}
  \mathcal{F}(\Phi_{t}) = \mathcal{E}(\Phi_{t}) + \log \left( \int e^{-\Phi_{t}} d \mu \right),
\end{equation*}
we obtain finally that
\begin{equation*}
  \int \dot{\Phi_{t}} e^{-\Phi_{t}} d \mu = 0.
\end{equation*}
But $\dot{\Phi_{t}} \geq 0$ since $t \mapsto \Phi_{t}$ is convex (by subharmonicity and $S^1$-invariance)
and therefore $\dot{\Phi_{t}} = 0$ almost everywhere. This leads to $\Phi_{0} = \Phi_{1}$.
\end{proof}


\section{Concluding remarks}
\label{sec:conclusion}


\subsection{The continuity method}
\label{subsec:continuity_method}

A classical strategy to solve $(MA)$ is to use
the \emph{continuity method}, looking at a continuous family of similar
Dirichlet problems,
\begin{equation*}
(MA)_{t} \hskip2cm
(dd^{c} \f_{t})^{n}=\frac{e^{-t\f_{t}} \mu}{\int_{\Omega} e^{-t \f_{t}} d \mu}
\; \text{ in } \Omega
\hskip.4cm
\text{ with }
\hskip.4cm
{\f_{t}}_{|\partial \Omega}=0,
\end{equation*}
where the parameter $t$ runs from $0$ to $1$. One sets
\begin{equation*}
  I:=\set{t \in [0,1] \, / \, (MA)_{t} \text{ admits a (smooth plurisubharmonic) solution}}
\end{equation*}
and then tries to show that $I$ is non empty, open and closed, so that $I=[0,1]$. Observe that
$1 \in I$ is equivalent to solving the Dirichlet problem $(MA)=(MA)_{1}$.

\medskip

It follows from the work of Cafarelli-Kohn-Nirenberg-Spruck \cite{CKNS85} that $0 \in I$, hence the latter is non empty (see the discussion in section~\ref{subsec:nonpositive}).

The a priori estimates derived in section~\ref{sec:higher_order_estimates} can be adapted to show that $I$ is closed. This is in general the most difficult part of the method. It however turns out  here that proving the openness is a delicate issue. Indeed, to do so, we need to show that the linearized $(MA)_{t}$ equation has a trivial kernel. More precisely, we have to prove that if $\varphi_{t}$ is a solution of $(MA)_{t}$, then every solution\footnote{In the following, covariant derivative, Ricci tensor and Laplacian referred to the metric defined by $\varphi_{t}$.} of
\begin{equation}\label{equ:linearized_MA}
  -\triangle\psi - t\psi + tc(\psi) = 0 \; \text{ in } \; \Omega \quad \text{with} \quad {\psi}_{|\partial \Omega}=0,
\end{equation}
where
\begin{equation*}
  c(\psi) := \int \psi (dd^{c} \f_{t})^{n}
\end{equation*}
must vanish. Let's introduce the differential operator
\begin{equation*}
  D : C^{\infty}(\Lambda^{0,1}\Omega) \to C^{\infty}(\Lambda^{0,1}\Omega \otimes\Lambda^{0,1}\Omega)
\end{equation*}
defined by
\begin{equation*}
  D\alpha := \nabla^{0,1}\alpha.
\end{equation*}
We have have then a \emph{Bochner formula} (up to an inessential multiplicative $\pi$ factor which we omit for brevity~\footnote{In the following computation $\triangle$ is the $\dbar$-Laplacian.})
\begin{equation}\label{eq:Bochner}
  -\triangle\alpha = D^{*}D\alpha + \Ric(\alpha), \qquad \alpha \in C^{\infty}(\Lambda^{0,1}\Omega).
\end{equation}
Applying \eqref{eq:Bochner} to $\dbar\psi$ where $\psi$ is a solution of \eqref{equ:linearized_MA}, we get
\begin{equation*}
  -\triangle \dbar \psi = t \dbar\psi =D^{*}D\dbar\psi + t \dbar\psi
\end{equation*}
because $\triangle$ and $\dbar$ commute and $\Ric(\alpha)=t\alpha$. Therefore
\begin{equation}\label{eq:integrant}
  D^{*}D\dbar\psi = 0.
\end{equation}
Then, taking the $L^{2}$ inner product of $D^{*}D\dbar\psi$ and $\dbar\psi$ and integrating by parts, without neglecting boundary terms (see \cite{GKY11} for details) and using the fact that on the boundary we have
\begin{equation*}
  \triangle\psi = tc(\psi),
\end{equation*}
we obtain
\begin{equation}\label{eq:Hessian_kernel}
   \norm{D\dbar\psi}_{L^{2}}^{2} = -\frac{1}{2} \int_{\partial \Omega} (n \cdot \psi)^{2}\Big[ \tr L_{\rho} + \Hess \rho (Jn,Jn) \Big] \, \sigma.
\end{equation}
where $\rho$ is a boundary defining function for $\partial \Omega$, $n$ is the outward unit normal vector field on
$\partial \Omega$ and $L_{\rho}$ is the \emph{Levi form} corresponding to $\rho$ (see section~\ref{subsec:clevi}).

If $\Omega$ is a strictly pseudoconvex domain then $\tr L_{\rho}$ is positive at each point of $\partial \Omega$, however we do not have \textit{a priori} any control on $\Hess \rho (Jn,Jn)$. So, contrary to what happens on a closed manifold where we do not have to deal with this disturbing boundary term, we cannot conclude here.

\begin{rem}
In the same spirit, we have shown in \cite{GKY11} that a ball of sufficiently large radius in complex projective space provides
an example of a strongly pseudoconvex domain which is not convex, and for which the \emph{Lichnerowicz estimate} fails.
\end{rem}


\subsection{Optimal constants}
\label{subsec:optimal_constant}

It is natural to wonder whether it is possible to solve
\begin{equation*}
(MA)_{t} \hskip2cm
(dd^{c} \f_{t})^{n}=\frac{e^{-t\f_{t}} \mu}{\int_{\Omega} e^{-t \f_{t}} d \mu}
\; \text{ in } \Omega
\hskip.4cm
\text{ with }
\hskip.4cm
{\f_{t}}_{|\partial \Omega}=0,
\end{equation*}
for
bigger values of $t>1$. As noticed in Remark~\ref{rem:biggervalues}, our Moser-Trudinger inequality
allows us to get control on slightly larger values of $t$, with a maximal value depending on $n$,
namely
\begin{equation*}
t<(2n)^{1+1/n}(1+1/n)^{(1+1/n)}.
\end{equation*}

It should be noticed that one can not expect to solve $(MA)_{t}$ for big values of $t$,
as follows from Bishop's volume comparison theorem.
Indeed, let $\mathbb{B}$ denote the unit ball in $\CC^{n}$. If we can find a solution $\f$ of $(MA)_{t}$ on $\mathbb{B}$, this means that we can find a K\"{a}hler-Einstein metric $\omega =dd^{c}\f$ on $\mathbb{B}$ satisfying $\Ric{(\omega )}=t\pi \omega$. Moreover, the volume $V$ of this metric is
\begin{equation*}
  V=\int_{\mathbb{B}} \frac{(dd^{c}\f )^{n}}{n!} = \frac{1}{n!}.
\end{equation*}
But by the Bishop volume comparison theorem, the volume has to be less than or equal to the volume of the $2n$-real dimensional sphere endowed with a metric of constant curvature $k$, with $k=(t\pi )/(2n-1)$. This implies that
\begin{equation*}
  \frac{1}{n!}\leq \frac{(4\pi )^{n}(n-1)!}{k^{n}(2n-1)!},
\end{equation*}
so that
\begin{equation*}
  t\leq 4(2n-1)\left[ \frac{(n-1)! n!}{(2n-1)!}\right]^{1/n}.
\end{equation*}
The interested reader will find in \cite{BB11} further motivation and references about $(MA)_{t}$ for large (critical) values of $t$.


\end{document}